\documentclass[a4paper]{amsart}

\usepackage{amsthm}
\usepackage{amsmath}
\usepackage{amssymb}
\usepackage{latexsym}
\usepackage{enumerate}
\usepackage{graphicx}
\usepackage[utf8]{inputenc}
\usepackage{epsfig} 
\usepackage{pgf}
\usepackage{tikz}
\usepackage{float}
\usepackage{times}

\newtheorem{theoremA}{Theorem}

\renewcommand{\thetheoremName}

\newtheorem{proposition[[]]}[theoremName]{Proposition G}

\newtheorem{theorem}{Theorem}[section]
\newtheorem{lemma}[theorem]{Lemma}

\newtheorem{proposition}[theorem]{Proposition}
\newtheorem{corollary}[theorem]{Corollary}

\theoremstyle{definition}
\newtheorem{definition}[theorem]{Definition}

\newtheorem{remark}{Remark}

\numberwithin{equation}{section}

\newcommand{\dist}{\operatorname{dist}}
\newcommand{\Vol}{\operatorname{Vol}}

\newcommand{\C}{\operatorname{Cap}}

\newcommand{\kam}{\mathbb{K}^{m}(b)}
\newcommand{\erre}{\mathbb{R}}

\begin{document}

\title[Ends, fundamental tones, and capacities]{Ends, fundamental tones, and capacities of minimal submanifolds via extrinsic comparison theory}
\author[V. Gimeno]{Vicent Gimeno}
\address{Departament de Matem\`{a}tiques- INIT-IMAC, Universitat Jaume I, Castell\'o,
Spain.}
\email{gimenov@uji.es}
\thanks{Work partially supported by DGI grant MTM2010-21206-C02-02.}

\author[S. Markvorsen]{S. Markvorsen}
\address{DTU Compute, Mathematics, Technical University of Denmark.}
\email{stema@dtu.dk}




\begin{abstract} We study the volume of extrinsic balls and the capacity of extrinsic annuli in minimal submanifolds which are properly immersed with controlled radial sectional curvatures into an ambient manifold with a pole. The key results are concerned with the comparison of those volumes and capacities with the corresponding entities in a rotationally symmetric model manifold. Using the asymptotic behavior of the volumes and capacities we then obtain upper bounds for the number of ends as well as estimates for the fundamental tone of the submanifolds in question. \end{abstract}

\maketitle

\section{Introduction}
Let $M$ be a complete non-compact Riemannian manifold. Let $K\subset M$ be a compact set with non-empty interior and smooth boundary. We denote by $\mathcal{E}_K(M)$  the number of connected components $E_1,\cdots , E_{\mathcal{E}_K(M)}$ of $M\setminus K$ with non-compact closure.  Then $M$ has  $\mathcal{E}_K(M)$ ends $\{E_i\}_{i=1}^{\mathcal{E}_K(M)}$ with respect to $K$ (see e.g. \cite{GriEnds}), and the \emph{global} number of ends $\mathcal{E}(M)$ is given by
 \begin{equation}
 \mathcal{E}(M)=\sup_{K\subset M} \mathcal{E}_K(M)\quad,
 \end{equation}
 where $K$ ranges on the compact sets of $M$ with non-empty interior and smooth boundary.
 
The number of ends of a manifold can be bounded by geometric restrictions. For example, in the particular setting of an $m-$dimensional minimal submanifold $P$ which is properly immersed into Euclidean space $\erre^n$, the number of ends $\mathcal{E}(P)$ is known to be related to the extrinsic properties of the immersion. Indeed, V. G. Tkachev proved in \cite[Theorem 2]{Tkachev}  (see also \cite{ChSubv}) that for any properly immersed $m-$dimensional minimal submanifold $P$ in $\erre^n$ with finite volume growth $V_{w_0}(P)<\infty$ the number of ends is bounded from above by
\begin{equation}\label{tka-ine}
\mathcal{E}(P)\leq C_m V_{w_0}(P)\quad,
\end{equation}
where  $C_m=1$ ($C_m={2^m}$ in the original \cite{Tkachev}) and the volume growth $V_{w_0}(P)$ is
\begin{equation}\label{tka-ine2}
V_{w_0}(P)=\lim_{R\to \infty}\frac{\Vol\left(P\cap B_R^{\erre^n}\left(o\right) \right)}{\Vol\left(B_R^{\erre^m}\left(o\right)\right)} \quad.
 \end{equation}
Here $\Vol\left(B_R^{\erre^m}\left(o\right)\right)$ is the volume of a geodesic ball $B_R^{\erre^m}\left(o\right)$ of radius $R$ centered at $o$ in $\erre^m$.  The inequality (\ref{tka-ine}) thus shows a significant relation between the number of ends (i.e. a topological property) and the behavior of a quotient of volumes (i.e. a metric property).
  
Motivated by Tkachev's application of the \emph{volume quotient} appearing in equation (\ref{tka-ine2}), we will consider the corresponding \emph{flux quotient} and \emph{capacity quotient} of the minimal submanifolds. These quotients are constructed in the same way as indicated by the \emph{volume quotient} but here we generalize the setting as well as Tkachev's result to minimal submanifolds in more general ambient spaces as alluded to in the abstract. Specifically we assume that the minimal immersion goes into an ambient manifold $N$ with a pole and with sectional curvatures $K_N$ bounded from above by the radial curvatures $K_w$ of a rotationally symmetric model space $M_w^n=\erre^+\times \mathbb{S}_1^{n-1}$, with warped metric tensor $g_{M_w^n}$ constructed  using a positive warping function $w:\erre^+\to \erre^+$ in such a way that  $g_{M_w^n}=dr^2+w(r)^2g_{\mathbb{S}_1^{n-1}}$ is also balanced from below (see \cite{MP2} and \S \ref{Prelim} for a precise definitions).
 
 Our generalization of inequality (\ref{tka-ine}) is thence the following:
  
 \begin{theorem}\label{flat-ends}
 Let $\varphi: P^m\to N^n$ be a proper minimal and complete immersion into a $n-$dimensional ambient manifold $N^n$ which possesses a pole $o\in N^n$ and its sectional curvatures $K_N$ at any point $p\in N$ are bounded by above by the radial curvatures $K_w$ of a balanced from below model space $M_w^n$
 \begin{equation}
 K_N\left(p\right) \leq K_{M_w^n}\left(r\left(p\right)\right)=-\frac{w''}{w}\left(r\left(p\right)\right)\quad.
  \end{equation}
  Suppose moreover, that $w'>0$ and there exist $R_0$ such that  $K_{M_w^n}(R)\leq 0$ for any $R>R_0$. Then, the number of ends $\mathcal{E}_{D_R}(P)$   with respect to the extrinsic ball $D_R=P\cap B_R^N(o)$ for $R>R_0$ is bounded from above by
  \begin{equation}\label{resul1}
  \mathcal{E}_{D_R}(P)\leq \left(\frac{2}{1-\frac{R}{t}} \right)^m\left(\frac{\int_0^tw(s)^{m-1}ds}{t^m/m}\right)\frac{\Vol(D_t)}{\Vol(B^w_t)} \quad,
  \end{equation}
  for any $t>R$.
 \end{theorem}
 
 Using the above theorem we can estimate the global number of ends as follows:
 
 \begin{corollary}\label{col-ends}
 Under the assumptions of theorem \ref{flat-ends}, suppose moreover
 \begin{equation}\label{ine-ends}
\limsup_{t\to\infty} \left(\frac{\int_0^tw(s)^{m-1}ds}{t^m/m}\right)=C_w<\infty \quad,
  \end{equation}
  and suppose also that the submanifold has finite volume growth, namely
 \begin{equation}
\Vol_{w}(P)=\lim_{t\to\infty}\frac{\Vol(D_t)}{\Vol(B^w_t)}<\infty \quad.
\end{equation}
Then
\begin{equation}\label{tka-gen-2}
  \mathcal{E}(P)\leq 2^m C_w\Vol_{w}(P)\quad.
\end{equation}
 \end{corollary}
 \begin{remark}
 If we choose $w(t)=w_0(t)=t$, the model space becomes $\erre^m$, which is balanced from below, and the hypothesis of theorem \ref{flat-ends} are therefore  automatically fulfilled for any complete minimal submanifold properly immersed in a Cartan-Hadamard ambient manifold.    Inequality (\ref{resul1}) becomes
 \begin{equation}
 \mathcal{E}_{D_R}(P)\leq \left(\frac{2}{1-\frac{R}{t}} \right)^m\frac{\Vol(D_t)}{V_mt^m} \quad,
  \end{equation}
  For any $R>0$ and any $t>R$, being $V_m$ the volume of a geodesic ball of radius $1$ in $\erre^m$. From inequality (\ref{ine-ends}) we get
  \begin{equation}
  C_{w_0}=1\quad.
  \end{equation}
  Thus inequality (\ref{tka-gen-2}) becomes
  \begin{equation}\label{tka-c-h}
  \mathcal{E}(P)\leq 2^m\lim_{t\to\infty}\frac{\Vol(D_t)}{V_mt^m} \quad,
  \end{equation}
  which is the original inequality obtained by Tkachev (inequality (\ref{tka-ine})), but now inequality (\ref{tka-c-h}) is valid for any minimal submanifold properly immersed in a Cartan-Hadamard ambient manifold with finite volume growth.
  \end{remark}

  In \cite{GPGap,ChSubv} are also obtained  lower bounds for the number of ends, but we note that those lower bounds seem to need stronger assumptions: Dimension greater than $2$, or embeddedness of the ends and codimension $1$, decay on the second fundamental form, and a rotationally symmetric ambient manifold. As a counterpart, those lower bounds are associated to the so-called gap type theorems.

  Combining the results of \cite[Theorem 3.5]{GPGap} and corollary \ref{col-ends}, and taking into account the role of sectional curvatures of the model space (see \cite[Proposition 2.6]{GPGap}) we have
  
\begin{corollary}Let $\varphi:P^m \to M^n_w$ be a minimal and proper immersion into a balanced from below model space $M^n_w$ with increasing warping function $w$ satisfying the following conditions:
  \begin{equation}
  \begin{aligned}
& \limsup_{t\to\infty} \left(\frac{\int_0^tw(s)^{m-1}ds}{t^m/m}\right)=C_w<\infty \quad,\\
&\text{ there exist }R_0 \text{ such that for any }R>R_0\\
  &
  \begin{cases}
    -\frac{w''(R)}{w(R)} \leq 0\\
  \frac{1-\left(w'(R)\right)^2}{w(R)^2}\leq 0
  \end{cases}
  \end{aligned}
   \end{equation}

Suppose moreover that $m>2$, that the center  $o_w$ of $M_w^n$  satisfies  $\varphi^{-1}(o_w)\neq \emptyset$  and that the norm of the second fundamental form $\Vert B^P\Vert $  of the immersion is bounded for large $r$ by
\begin{equation}
\Vert B^P\Vert \leq \frac{\epsilon(r)}{w'(r)w(r)}\quad,
 \end{equation}
where $\epsilon$ is a positive function such that $\epsilon(r)\to 0$ when $r\to \infty$.

  Then the number of ends is bounded from below and from above by
\begin{equation}
  \Vol_{w}(P)\leq  \mathcal{E}(P)\leq 2^m C_w\Vol_{w}(P)\quad.
\end{equation}
   \end{corollary}
 
\begin{figure}\label{figure}
\begin{center}
\includegraphics[width=25mm]{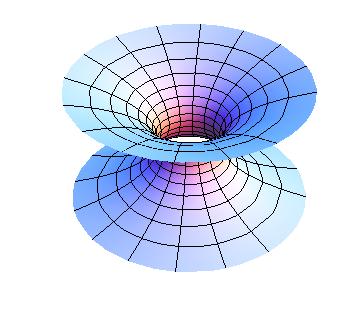}\quad\includegraphics[width=15mm]{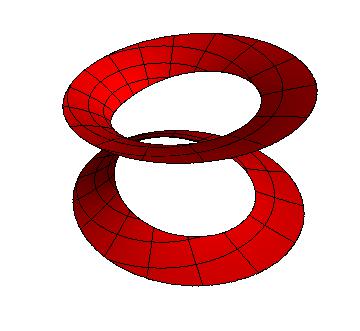}\quad \includegraphics[width=25mm]{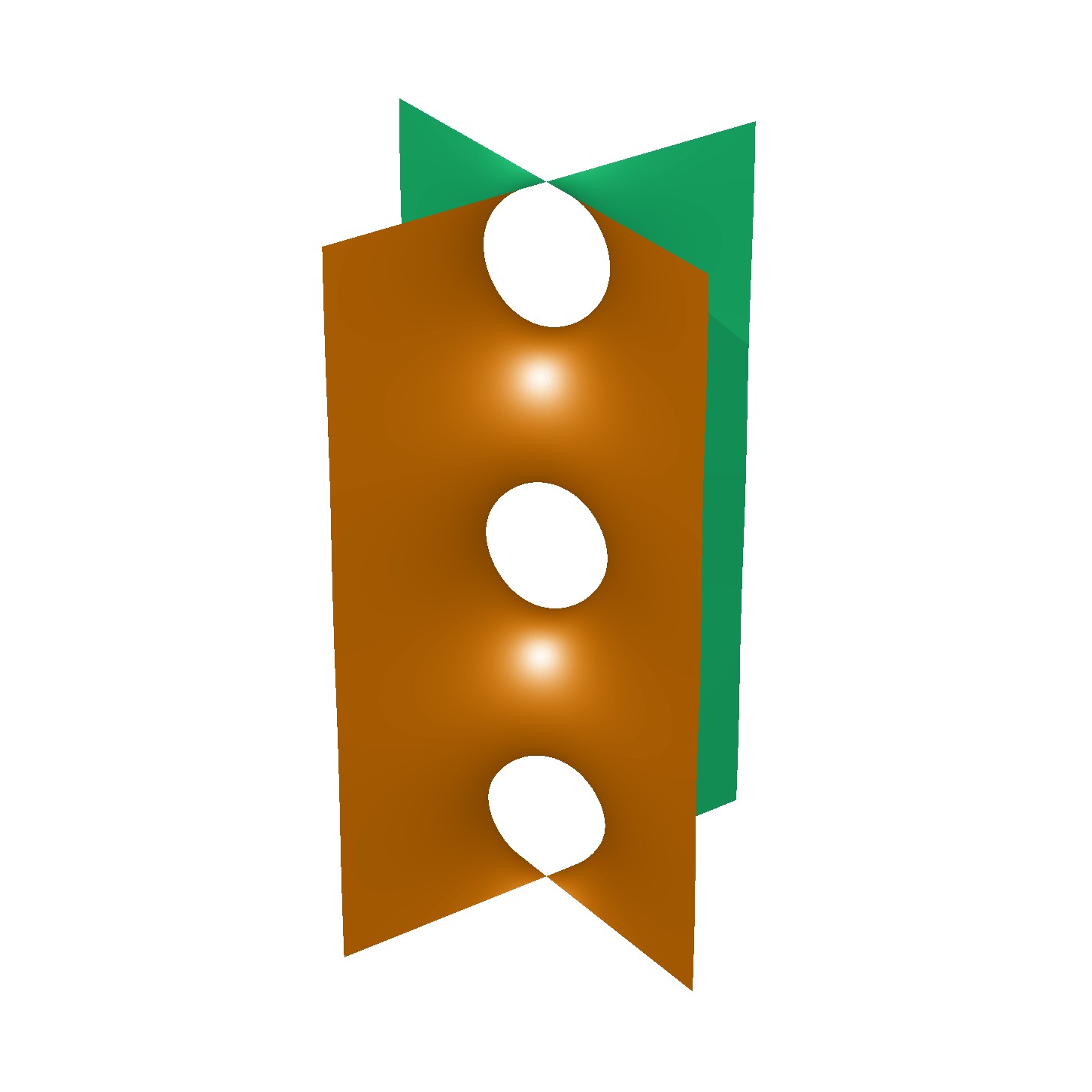} \quad \includegraphics[width=25mm]{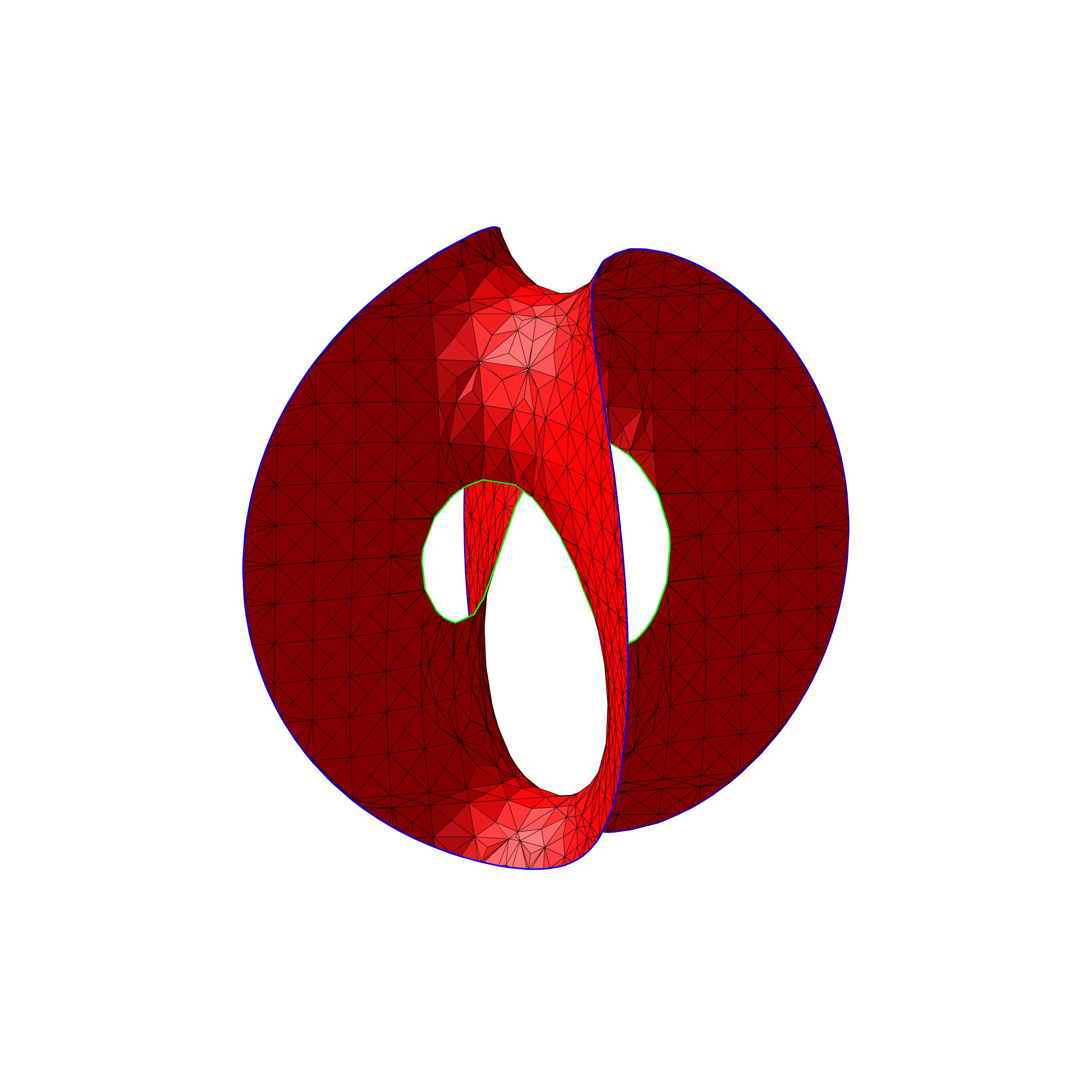}
\end{center}
\caption{Two examples of extrinsic annuli in $\erre^3$: A catenoid on the left  and the singly periodic Scherk surface on the right. The extrinsic annuli are constructed by cutting the surfaces with two spheres (with the same center but of  different radii) in the ambient manifold ($\erre^3$). The catenoid has two ends and finite total curvature. Hence, by theorem \ref{intro-theo-2}, the capacity of the extrinsic annulus of the catenoid is greater than the capacity of the corresponding annulus of the Euclidean $2$-plane but is smaller than two times that capacity.  The same is true for the extrinsic annulus of the singly periodic Scherk surface (we refer the reader to the introduction of \cite{Meeks2007} for the area growth of the singly periodic Scherk surface). }
 \end{figure}

By using our results about the behavior of the comparison quotients we can also estimate the capacity of an extrinsic annulus $\text{A}_{\rho, R}=P\cap \left(B_R^N(o)\setminus B_\rho^N(o)\right)$ (see figure \ref{figure} and, \S \ref{Results} and \S \ref{Prelim} for a precise definition of \emph{capacity} and \emph{extrinsic annulus}):
  \begin{theorem}\label{intro-theo-2}
  Let $\varphi: P^m\to \erre^n$ denote a complete and proper minimal immersion into the Euclidean space $\erre^n$. Then, for any $R>\rho>0$, the capacity of the extrinsic annulus $A_{\rho,R}$ is bounded by
  \begin{equation}\label{ine-intr-theo-2}
  \frac{\Vol(D_\rho)}{V_m \rho^m}\leq \frac{\C(A_{\rho,R})}{\C(A^{\erre^m}_{\rho,R})}\leq  \frac{\Vol(D_R)}{V_m R^m}\quad,
   \end{equation}
  where $\C(A^{\erre^m}_{\rho,R})$ is the capacity of the geodesic annulus $A^{\erre^m}_{\rho,R}$ in $\erre^m$ of inner radius $\rho$ and outer radius $R$.
  \end{theorem}
  
  \begin{remark}
  Since, from Theorem \ref{isoperimetric}, the quotient $\frac{\Vol(D_s)}{V_m s^m}$ is a non-decreasing function on $s$, we can state Theorem \ref{intro-theo-2} in the limit case ( $\rho\to 0$ and $R\to\infty$) and inequality (\ref{ine-intr-theo-2}) there becomes
  \begin{equation}
  1\leq \frac{\C(A_{\rho,R})}{\C(A^{\erre^m}_{\rho,R})}\leq \lim_{R\to\infty}\frac{\Vol(D_R)}{V_m R^m}= V_{w_0}(P)\quad.
   \end{equation}
When we deal with  a minimal surface $\Sigma\subset\erre^3$ which is properly \emph{embedded} into the Euclidean space $\erre^3$ the limit
\begin{equation}\label{two-d}
V_{w_0}(\Sigma)=\lim_{R\to\infty}\frac{\Vol(\Sigma\cap B_R^{\erre^3}(o))}{\pi R^2}
\end{equation}
is well understood. For instance, the above limit corresponds to the number of ends if the surface $\Sigma$ has finite total curvature.
  \end{remark}
  
  \begin{remark}
  In order to bound the capacity quotient, our theorems do not make use of the volume quotient as in Theorem \ref{intro-theo-2}, but instead they make use of the flux quotient (see Theorems \ref{capacity-baix} and \ref{capacity}). In the special case when the ambient manifold is $\erre^n$ (such as in Theorem \ref{intro-theo-2}) the volume quotient agrees however with the flux quotient (see equation (\ref{equal-flux-vol}) and theorem \ref{volume-flux-theo}).
  \end{remark}
  
   \subsection{Outline of the paper}
   In \S \ref{Results} we show our main theorems concerning the flux quotients, the volume quotients, and the capacity quotients. In \S \ref{Prelim} we state the preliminary concepts in order to prove  the main theorems of \S \ref{Results} in \S \ref{proof-main}. This allows us then to prove Theorem \ref{flat-ends} and Corollary \ref{col-ends} in \S\ref{Proof-into}. Finally, in \S \ref{corollaries}, we present several corollaries and examples of applications of the extrinsic theory and results which have been established in \S  \ref{Results}.

\section{Extrinsic theory: Flux, Capacity and Volume comparison for extrinsic balls }\label{Results}
Let $(M^n,g)$ be a Riemannian  manifold. For any oriented
hypersurface $\Sigma\subset M$ with unit normal vector field $\nu$, we define the flux $F_X(\Sigma)$ of the vector field $X$ through $\Sigma$ by
\begin{equation}\label{flux}
F_X(\Sigma):=\int_\Sigma\langle X, \nu\rangle d\mu_\Sigma\quad ,
\end{equation}
where $d\mu_\Sigma$ is the associated Riemannian density determined by the metric $g_\Sigma=i^*g$ (being $i:\Sigma \to M$ the inclusion map).

By the divergence theorem (see \cite{Chavel2} for instance), if one has an oriented domain $\Omega$ in $M$ with smooth boundary $\partial \Omega$, and the vector field $X$ is $C^1$ in $\overline \Omega$ and with compact support in $\overline \Omega$, the flux of $X$ through $\partial \Omega$ is related to the divergence of $X$ by
\begin{equation}\label{divergence}
\int_\Omega\text{div} X d\mu=\int_{\partial \Omega}\langle X,
\nu\rangle d\mu_{\partial \Omega}=F_X(\partial \Omega)\quad .
\end{equation}

Given a smooth function $u:M\to \erre$, we can also define the flux of a function $u$, but then the flux $J_u(t)$ is the flux of the gradient $\nabla u$ (i.e. the metric dual vector to $du$, $du(X)=\langle \nabla u, X\rangle$)  through  the level set $\Sigma^u_t:=\{x\in M\,\vert\, u(x)=t\}$ so that:

\begin{equation}
J_u(t):=F_{\nabla u}(\Sigma_t^u)\quad .
\end{equation}

Taking into account that the unit normal vector field $\nu$ of $\Sigma^u_t$ is $\nu=\frac{\nabla u}{\vert \nabla u \vert}$, it is easy to see that 

\begin{equation}\label{flux-function}
J_u(t)=\int_{\Sigma_t}\vert \nabla u \vert d\mu_{\Sigma_t^u}\quad.
\end{equation}

Observe moreover, that by the Sard theorem and by the regular set theorem we need no further restrictions on the smoothness of $\Sigma_t^u$ and on the smoothness of the unit normal vector field $\nu$.

The overall goal of this work is to characterize the \emph{isoperimetric inequalities for extrinsic balls}, and the \emph{capacity of minimal submanifolds} in terms of the flux of extrinsic distance functions.  Actually we are interested on the flux of the extrinsic distance function on minimal submanifolds in an ambient manifold $N$ which possesses
a pole and has the radial curvatures bounded form above by the
radial curvatures of rotationally symmetric model space $K_N\leq
K_{M_w^n}=-\frac{w''}{w}$, see \cite{MP2} or section \ref{Prelim}
of this paper for precise definitions. It is the behavior of this particular flux that allows us  to study the mean exit time function, the capacity, the conformal type, the fundamental tone, and in special cases also the number of ends of the submanifold.

\subsection{Flux and volume comparison: isoperimetric inequalities
  and the mean exit time function}

Given an isometric immersion $\varphi:P\to (N,o)$ into a manifold with
a pole $o\in N$, the flux $J_r$ of the
extrinsic distance function $r_o$ ({\it i.e.}, the restriction by the
immersion of the ambient distance function to the submanifold) is given by
$$
J_r(R)=\int_{\partial D_R}\vert \nabla^P r_o\vert d\mu\quad ,
$$
where $\partial D_R$ is the level set $\partial D_R=r_o^{-1}(R)$, and therefore, $D_R=r_o^{-1}([0,R))$ is the extrinsic ball of radius $R$.

When the immersion is minimal and the ambient  manifold has its  radial sectional curvatures $K_N$ bounded
from above by the radial sectional curvatures of a
rotationally symmetric model space $M_w^n$ that is balanced from below
(see \cite{MP2} and section \ref{Prelim}), $K_N\leq K_{M_w^n}$, we can compare the
volume quotient $\frac{\Vol(D_R)}{\Vol(B_R^w)}$ and the flux quotient $\frac{J_r(R)}{J_r^w(R)}$ . The volume quotient  is the quotient between the volume of a extrinsic ball $D_R$ of radius $R$ in $P^m$ and
the volume of a geodesic ball $B_R^w$ of the same radius $R$ in $M_w^m$. The flux quotient is the
quotient   between the flux of the extrinsic distance in $P^m$ and the
flux of the geodesic distance in $M_w ^m$.  These two quotients are related by the following theorem

\begin{theorem}\label{isoperimetric}
Let $\varphi: P^m \longrightarrow N^n$ be an isometric, proper, and minimal immersion of a complete non-compact Riemannian $m$-manifold $P^m$ into a complete Riemannian manifold $N^n$ with a pole $o\in N$ .  Let us suppose that the $o-$radial sectional curvatures of $N$ are bounded from above by
$$K_{o,N}(\sigma_{x}) \leq  -\frac{w''(r)}{w(r)}(\varphi(x))\,\,\,\forall x \in P\, ,$$ and the model space $M_w^m$ is balanced from bellow. Then

\begin{enumerate}
\item $J_r(R)$ is related with $\Vol(D_R)$ by
\begin{equation}\label{equ1}
 \frac{\Vol(D_R)}{\Vol(B_R^w)}\leq \frac{J_r(R)}{J_r^w(R)}.
\end{equation}

\item The functions $\frac{\Vol(D_R)}{\Vol(B_R^w)}$ and  $\frac{J_r(R)}{J_r^w(R)}$ are non decreasing functions on $R$.

\item Denoting by $E_R^P(x)$ the mean time for the first exit from the extrinsic ball $D_R(o)$ for a Brownnian particle starting at $o\in P^m$,  and denoting by $E_R^w$  the mean exit time function for the $R-$ball $B_R^w$ in the model space $M_w^m$, if equality holds in (\ref{equ1}) for some fixed radius $R>0$, then for any $x\in D_R$,  $E_R^P(x)=E_R^w(r(x))$, where  $r(x)$ the extrinsic distance from $o$ to the point $x\in P$.
\end {enumerate}
\end{theorem}

\subsection{Capacity and flux comparison: conformal type}

Given a compact set $F\subset M$ in a Riemannian manifold $M$ and an open set $G\subset M$ containing $F$, we call the couple $(F,G)$ a \emph{capacitor}. Each capacitor then has its capacity defined by

\begin{equation}
\C(F,G):=\inf_u\int_{G\setminus F}\Vert \nabla u \Vert
d\mu\quad ,
\end{equation}
where the $\inf$ is taken over all Lipschitz functions $u$ with compact support in $G$ such that $u=1$ on $F$.

When $G$ is precompact, the infimum is attained for the function $u=\Psi$ which is the solution of the following Dirichlet problem in $G-F$:
\begin{equation}\label{Dir-capa}
\begin{cases}
\Delta \Psi=0\\
\Psi\vert_{\partial F}=0\\
\Psi\vert_{\partial G}=1
\end{cases}
\end{equation}

From a physical point of view, the capacity of the capacitor $(F,G)$ represents the total electric charge (generated by the electrostatic potential $\Psi$) flowing into the domain $G-F$ through the interior boundary $\partial F$.
Since the total current stems from a potential difference of $1$ between $\partial F$ and $\partial G$, we get from Ohm's Law that the effective resistance of the domain $G-F$ is
\begin{equation}
R_{\text{eff}}(G-F)=\frac{1}{\C(F,G)}\quad .
\end{equation}

The exact value of the capacity of a set is known only in a few cases, and so its estimation in geometrical terms is of great interest, not only in electrostatic, but in many physical descriptions of flows, fluids, heat, or generally where the Laplace operator plays a key role, see \cite{conj-ps,HPR}.

Given a capacitor  $(F,G)$, if we have a smooth function $u$ with $u=a$
on $\partial F$ and $u=b$ on $\partial G$. The capacity and the flux
are then related by     (see \cite{GriCap}):
\begin{equation}\label{capa-gry}
\C(F,G)\leq \left(\int_a^b \frac{ds}{J_u(s)}\right)^{-1}\quad .
\end{equation}

In this paper we are interested on the $o$-centered \emph{extrinsic annulus} $A_{\rho,R}(o)\subset P^m$
for $0<\rho<R$  given by
\begin{equation}
A_{\rho,R}(o):=D_R(o)-D_\rho(o)\quad.
\end{equation}
To be more precise, we are interested on the behavior of the flux and the capacity of those extrinsic domains. In the following theorems we provide upper and lower bounds for the capacity quotient in terms of the flux quotient.

\begin{theorem}\label{capacity-baix}Let $\varphi: P^m \longrightarrow N^n$ be an isometric, proper, and minimal immersion of a complete non-compact Riemannian $m$-manifold $P^m$ into a complete Riemannian manifold $N^n$ with a pole $o\in N$ and satisfying $\varphi^{-1}(o) \neq \emptyset$.  Let us suppose that the $o-$radial sectional curvatures of $N$ are bounded from above by
$$K_{o,N}(\sigma_{x}) \leq
-\frac{w''(r)}{w(r)}(\varphi(x))\,\,\,\forall x \in P \quad ,$$ and
the warping function $w$ satisfies
$$
w'\geq 0\quad .
$$
Then
\begin{equation}
\frac{J_r(\rho)}{J_r^w(\rho)} \leq
\frac{\C(A_{\rho,R})}{\C(A_{\rho,R}^w)} \quad ,
\end{equation}
where $A^w_{\rho,R}$ is the intrinsic annulus in $M_w^m$.
\end{theorem}

\begin{theorem}\label{capacity}Let $\varphi: P^m \longrightarrow N^n$ be an isometric, proper, and minimal immersion of a complete non-compact Riemannian $m$-manifold $P^m$ into a complete Riemannian manifold $N^n$ with a pole $o\in N$ .  Let us suppose that the $o-$radial sectional curvatures of $N$ are bounded from above by
$$K_{o,N}(\sigma_{x}) \leq  -\frac{w''(r)}{w(r)}(\varphi(x))\,\,\,\forall x \in P\, ,$$ and the model space $M_w^m$ is balanced from bellow. Then

\begin{equation}\label{equ2}
\frac{\C(A_{\rho,R})}{\C(A_{\rho,R}^w)}\leq \frac{J_r(R)}{J_r^w(R)} \quad ,
\end{equation}
where $A^w_{\rho,R}$ is the intrinsic annulus in $M_w^m$. Moreover, if equality holds in (\ref{equ2}) for some fixed $R>0$, then $D_R$ is a minimal cone in $N^n$.
\end{theorem}

Geometric estimates of the capacity are sufficient to obtain large scale
consequences such as as the parabolic or hyperbolic character of the
manifold, \cite{I2,I1, MPcapacity,MP1}. We note here the following important equivalent conditions about the conformal type:

\begin{theoremA}\label{sullivan}
Let (M,g) be a given Riemannian manifold, Then the following conditions are equivalent
\begin{itemize}
\item There is a precompact open domain $K$ in $M$, such that the Brownian motion $X_t$ starting from $K$ does not return to $K$ with probability $1$, i.e.
\begin{equation}
P_x\left\{\omega \vert X_t(\omega) \in K \text{ for some }t>0 \right\}<1
\end{equation}
\item $M$ has positive capacity: There exist in $M$ a compact domain $K$, such that
\begin{equation}
\C(K,M)>0
\end{equation}
\item $M$ has finite resistance to infinity: There exist in $M$ a compact domain $K$, such that
\begin{equation}
R_\text{eff}(M-K)<\infty
\end{equation}
\end{itemize}
\end{theoremA}

A manifold satisfying the conditions of the above theorem will be called a
hyperbolic manifold, otherwise it is called a parabolic manifold.

As a consequence of the above theorem we can state the following corollary for minimal submanifolds:

 \begin{corollary}\label{cor-capacity-baix}Let $\varphi: P^m \longrightarrow N^n$ be an isometric, proper, and minimal immersion of a complete non-compact Riemannian $m$-manifold $P^m$ into a complete Riemannian manifold $N^n$ with a pole $o\in N$ .  Let us suppose that the $o-$radial sectional curvatures of $N$ are bounded from above by
$$K_{o,N}(\sigma_{x}) \leq
-\frac{w''(r)}{w(r)}(\varphi(x))\,\,\,\forall x \in P \quad ,$$ and
the warping function $w$ satisfies
$$
w'\geq 0\quad .
$$
Then

\begin{enumerate}
\item If $M_w^m$ is a hyperbolic manifold, then $P$ is a hyperbolic manifold.
\item In consequence, if $P$ is parabolic, then $M_w^m$ is also parabolic.
\end{enumerate}
\end{corollary}

Since $\frac{J_r(R)}{J_r^w(R)} $ and $\frac{\Vol(D_R)}{\Vol(B_R^w)}$ are  non-decreasing functions under our hypothesis, we can define  two expressions which are analogous to the projective volume defined by V. G. Tkachev in \cite{Tkachev}

\begin{definition}
Given $\varphi:P^m\to N^n$ an immersion into a manifold $N$ with a pole $o\in N$. The  $w$-\emph{flux} $\text{Flux}_w(P)$ and the  $w$-\emph{volume}  $\Vol_w(P)$ of the submanifold $P$ are defined by :
\begin{equation}
\begin{aligned}
\text{Flux}_w(P)&:=\sup_{R\in \erre^+}\frac{J_r(R)}{J_r^w(R)}\quad , \\
\Vol_w(P)&:=\sup_{R\in \erre^+}\frac{\Vol(D_R)}{\Vol(B^w_R)}\quad .
\end{aligned}
\end{equation}
We will say that $P$ has \emph{finite} $w-$ \emph{flux} ( resp. \emph{finite} $w-$ \emph{volume})  if and only if  $\text{Flux}_w(P)<\infty$  ( or $\Vol_w(P)<\infty$).
\end{definition}
We refer to theorem \ref{volume-flux-theo} for the relation between the $w-$flux and the $w-$volume of a submanifold.

From theorem \ref{sullivan} and theorem \ref{capacity} we can  now state that for minimal submanifolds with finite $w-$flux we have:

\begin{corollary}\label{parabolic-hiperbolic-dalt} Let $\varphi: P^m \longrightarrow N^n$ be an isometric, proper and minimal immersion of a complete non-compact Riemannian $m$-manifold $P^m$ into a complete Riemannian manifold $N^n$ with a pole $o\in N^n$.  Let us suppose that the $o-$radial sectional curvatures of $N^n$ are bounded from above as follows
$$
K_{o,N}(\sigma_{x}) \leq
-\frac{w''(r)}{w(r)}(\varphi(x))\,\,\,\forall x \in P \quad ,
$$
and that the model space $M_w^m$ is balanced from below. Suppose
moreover that $P$ has finite $w-$flux.  Then
\begin{enumerate}
 \item If $M_w^m$ is a parabolic manifold, then $P$ is a parabolic
   manifold.
\item If $P$ is an hyperbolic manifold, then $M_w^n$ is an hyperbolic manifold.
\end{enumerate}
\end{corollary}

Joining the previous two corollaries together we get:

\begin{corollary}\label{parabolic-hiperbolic}Let $\varphi: P^m \longrightarrow N^n$ be an isometric, proper and minimal immersion of a complete non-compact Riemannian $m$-manifold $P^m$ into a complete Riemannian manifold $N^n$ with a pole $o\in N^n$ .  Let us suppose that the $o-$radial sectional curvatures of $N^n$ are bounded from above,
$$
K_{o,N}(\sigma_{x}) \leq
-\frac{w''(r)}{w(r)}(\varphi(x))\,\,\,\forall x \in P\quad ,
$$
that the warping function $w$ satisfies
$$
w'\geq 0\quad,
$$
and that the model space $M_w^m$ is balanced from below,  and that $P$ has finite $w-$flux.  Then $P$ is hyperbolic (parabolic) if and only if $M_w^m$ is hyperbolic (parabolic).
\end{corollary}




\section{Preliminaires}\label{Prelim}

We assume throughout the paper that $\varphi: P^m \longrightarrow N^n$ is an isometric immersion of a complete non-compact Riemannian $m$-manifold $P^m$ into a complete Riemannian manifold $N^n$ with a pole $o\in N^n$ . Recall that a pole
is a point $o$ such that the exponential map
$$\exp_{o}\colon T_{o}N^{n} \to N^{n}$$ is a
diffeomorphism.

  For every $x \in N^{n}- \{o\}$ we
define $r(x) = r_o(x) = \dist_{N}(o, x)$, since $o$ is a pole  this
distance is realized by the length of a unique
geodesic from $o$ to $x$, which is the {\it
radial geodesic from $o$}. We also denote by $r\vert_P$ or by $r$
the composition $r\circ \varphi: P\to \erre_{+} \cup
\{0\}$. This composition is called the
{\em{extrinsic distance function}} from $o$ in
$P^m$.

With the extrinsic distance we can construct the {\em extrinsic ball} $D_R(o)$ of radius $R$ centered at $o$ as

$$
D_R(o):=\{x\in P : r(\varphi(x))< R\}\quad.
$$
Since $\partial D_t(o)=\Sigma_t^r$, the flux of the extrinsic distance function $r$ on $P$ is
$$
J_r(t)=\int_{\partial D_t}\vert \nabla^P r\vert d\rho \quad ,
$$

where the gradients of $r$ in $N$ and $r\vert_P$ in  $P$ are
denoted by $\nabla^N r$ and $\nabla^P r$,
respectively. These two gradients have
the following basic relation, by virtue of the identification, given any point $x\in P$, between the tangent vector fields $X \in T_xP$ and $\varphi_{*_{x}}(X) \in T_{\varphi(x)}N$

\begin{equation}\label{radiality}
\nabla^N r = \nabla^P r +(\nabla^N r)^\bot ,
\end{equation}
where $(\nabla^N r)^\bot(\varphi(x))=\nabla^\bot r(\varphi(x))$ is perpendicular to
$T_{x}P$ for all $x\in P$.

We now present the curvature restrictions which constitute the geometric framework of the present study.

\begin{definition}
Let $o$ be a point in a Riemannian manifold $N$
and let $x \in N-\{ o \}$. The sectional
curvature $K_{N}(\sigma_{x})$ of the two-plane
$\sigma_{x} \in T_{x}N$ is then called a
\textit{$o$-radial sectional curvature} of $N$ at
$x$ if $\sigma_{x}$ contains the tangent vector
to a minimal geodesic from $o$ to $x$. We denote
these curvatures by $K_{o, N}(\sigma_{x})$.
\end{definition}

\subsection{Model spaces}\label{subsecWarp}
Throughout this paper we shall assume that the ambient manifold
$N^n$ has its $o$-radial sectional curvatures $K_{o,N}(x)$ bounded
from above by the expression $K_w(r(x))=-w''(r(x))/w(r(x))$, which are
precisely the radial sectional curvatures of the {\em
$w$-model space} $\,M^{m}_{w}\,$ we are going to define.

\begin{definition}[See {\cite{Oneill,GriExp,GW}}]\label{defModel}
A $w-$model $M_{w}^{m}$ is a smooth warped product with base $B^{1}
= [0,\Lambda[ \,\subset \mathbb{R}$ (where $0 < \Lambda
\leq  \infty$), fiber $F^{m-1} = \mathbb{S}^{m-1}_{1}$ (i.e. the unit
$(m-1)$-sphere with standard metric), and warping function $w\colon
[0,\Lambda[ \to \mathbb{R}_{+}\cup \{0\}$, with $w(0) = 0$,
$w'(0) = 1$, and $w(r) > 0$ for all $r >  0$. The point
$o_{w} = \pi^{-1}(0)$, where $\pi$ denotes the projection onto
$B^1$, is called the {\em{center point}} of the model space. If
$\Lambda = \infty$, then $o_{w}$ is a pole of $M_{w}^{m}$.
\end{definition}

\begin{proposition}\label{propSpaceForm}
The simply connected space forms $\mathbb{K}^{m}(b)$ of constant
curvature $b$ are $w-$models with warping functions
\begin{equation*}
w_b(r) = \begin{cases} \frac{1}{\sqrt{b}}\sin(\sqrt{b}\, r) &\text{if $b>0$}\\
\phantom{\frac{1}{\sqrt{b}}} r &\text{if $b=0$}\\
\frac{1}{\sqrt{-b}}\sinh(\sqrt{-b}\,r) &\text{if $b<0$}.
\end{cases}
\end{equation*}
Note that for $b > 0$ the function $w_{b}(r)$ admits a smooth
extension to  $r = \pi/\sqrt{b}$.
\end{proposition}

\begin{proposition}[See {\cite{Oneill,GW,GriExp}}]\label{propWarpMean}
Let $M_{w}^{m}$ be a $w-$model space with warping function $w(r)$ and
center $o_{w}$. The distance sphere of radius $r$ and center $o_{w}$
in $M_{w}^{m}$ is the fiber $\pi^{-1}(r)$. This distance sphere has
the constant mean curvature $\eta_{w}(r)= \frac{w'(r)}{w(r)}$ On the other hand, the
$o_{w}$-radial sectional curvatures of $M_{w}^{m}$ at every $x \in
\pi^{-1}(r)$ (for $r > 0$) are all identical and determined
by
\begin{equation*}
K_{o_{w} , M_{w}}(\sigma_{x}) = -\frac{w''(r)}{w(r)}.
\end{equation*}
\end{proposition}

\begin{remark}
The $w-$model spaces are completely determined via $w$
by the mean curvatures of the spherical fibers $S^{w}_{r}$:
$$
\,\eta_{w}(r) = w'(r)/w(r)\,\quad,
$$
by the
volume of the fiber
$$
\,\Vol(S^{w}_{r}) \, = V_{0}\,w^{m-1}(r)\,\quad ,
$$
and by the volume of the corresponding ball, for which the fiber is
the boundary
$$
\,\Vol(B^{w}_{r}) \, = \, V_{0}\, \int_{0}^{r}
\,w^{m-1}(t)\,dt\,\quad .
$$
Here $V_{0}$ denotes the volume of the unit sphere $S^{0,m-1}_{1}$, (we denote in general as $S^{b,m-1}_r$ the sphere of radius $r$ in the real space form $\kam$) .
The latter two functions define the isoperimetric quotient function
as follows
$$
\,q_{w}(r) \, = \, \Vol(B^{w}_{r})/\Vol(S^{w}_{r}) \quad .
$$
We observe moreover that the flux of the geodesic distance function
$r_o$ from the center to the model space is
$$
J^w_r(R)=\int_{S^w_R}\vert \nabla r \vert d\sigma=\Vol(S^w_R)\quad.
$$
\end{remark}

Besides the already defined comparison controllers for the radial sectional curvatures of
$N^{n}$, we shall need two further purely intrinsic conditions
on the model spaces:

\begin{definition}
A given $w-$model space $\, M^{m}_{w}\,$ is
called balanced from below and balanced from above, respectively, if
the following weighted isoperimetric conditions are satisfied:
$$
\begin{aligned}
\text{Balance from below:}\quad q_{w}(r)\,\eta_{w}(r) \, &\geq 1/m \quad \text{for all} \quad r \geq 0 \quad ;\\
\text{Balance from above:}\quad  q_{w}(r)\,\eta_{w}(r) \, &\leq
1/(m-1) \quad \text{for all} \quad r \geq 0 \quad .
\end{aligned}
$$
A model space is called {\em{totally balanced}} if it is balanced
both from below and from above.
\end{definition}
\subsection{Laplacian comparison for radial functions}
Let us recall the expression of the Laplacian on model spaces for
radial functions
\begin{proposition} [See \cite{Oneill},  \cite{GW} and \cite{GriExp}]
Let $M_w^n$ be a model space, denote by $r:M_w^n-\{o_w\}\to \erre^+$ the
geodesic distance from the center $ o_w$, let $f:\erre \to \erre$ be a smooth function, then
\begin{equation}
\Delta^{M_w^n}\left(f\circ r\right)=f''\circ r+\left(n-1\right)
\left(f'\cdot \eta_w\right)\circ r\quad .
\end{equation}
\end{proposition}

Applying the Hessian comparison theorems given in \cite{GW} we can obtain (see
\cite{MP2} for instance)
\begin{proposition}\label{lap-comp} Let $\varphi:P^m\to N^n$ be an immersion into a
  manifold $N$ with a pole. Suppose the the radial sectional
  curvatures $K_N$ of $N$ are bounded from above by the radial sectional
  curvatures of a model space $M_w^m$ as follows:
\begin{equation}
K_N\leq -\frac{w''}{w}\circ r\quad .
\end{equation}
Let $f:\erre \to \erre$ be a smooth function with $f'\geq 0$, and
dennote by $r:P\to \erre^+$ the extrinsic distance function. Then
\begin{equation}
\begin{aligned}
\Delta^P\left(f\circ r\right)\geq &\vert \nabla^P r\vert\left(f''-f'\cdot \eta_w\right)\circ
r \\ &+ m \left(f'\cdot \eta_w\right)\circ r+m\langle \nabla^N
r, H_P\rangle f'\circ r\quad,
\end{aligned}
\end{equation}
where $H_P$ denotes the mean curvature vector of $P$ in $N$.
\end{proposition}

\subsubsection{Capacity and the Mean Exit Time function on Model spaces}

One key purpose of this paper is to compare the capacity of extrinsic annuli of an immersed minimal submanifold with the capacity in an adequate model space. In the model space we can obtain the value of the capacity directly:

\begin{proposition}[See \cite{GriCap}] Let $M_w^n$ be a model space. Then
\begin{equation}
\C(A^w_{\rho,R})=\left(\int_\rho^R\frac{ds}{\Vol(S_s^w)}\right)^{-1}=V_n \left(\int_\rho^R\frac{ds}{w^{n-1}}\right)^{-1}\quad .
\end{equation}
\end{proposition}

We note that the radial function $\Psi:M_w^m\to \erre$ given by
\begin{equation}
\Psi(p):=\Psi^w_{\rho,R}(r(p)) \quad ,
\end{equation}
being
\begin{equation}
\Psi^w_{\rho,R}(t)=\int_\rho ^t \frac{\C(A^w_{\rho,R})}{\Vol(S^w_s)}ds\quad,
\end{equation}
is the solution to the Dirichlet problem given in \ref{Dir-capa} for
the annular region $A^w_{\rho, R}$, namely
\begin{equation}
\begin{cases}
\Delta^{M_w^m} \Psi=0\\
\Psi\vert_{S^w_\rho}=0\\
\Psi\vert_{S^w_R}=1
\end{cases}
\end{equation}

Another important tool in this paper is the comparison result for the mean exit time. Let now $E_R^w$ denote the mean time of the first exit from $B_R^w$ for a Brownnian particle starting at $o_w$. A remark due to Dynkin in \cite{Dynkin} claims that $E_R^w$ is the continuous solution to the following Poisson equation with Dirichlet boundary data,
\begin{equation}
\begin{aligned}
\Delta^{M_w^n}E_R^w=&-1\\
E_R^w\vert_{S_R^w}=& 0.
\end{aligned}
\end{equation}

Since the ball $B_R^w$ has maximal isotropy at the center $o_{w}$, so we have that $E_R^w$ only depends on the extrinsic distance $r$. Therefore, we will write $E_R^w=E_R^w(r)$ and

\begin{proposition}[See \cite{MP2}] Let $M_w^n$ be a model space of dimension $n$ then
\begin{equation}
E^w_R(r)=\int_r^Rq_w(t)dt
\end{equation}
\end{proposition}

\section{Proof of the main theorems of \S \ref{Results}}\label{proof-main}
\subsection{Proof of theorem \ref{isoperimetric}}

Since the mean time function $E_R^w$ is a radial function, we can transplant it to $P$ using the extrinsic distance, hence, we also denote as $E_R^w:P\to \erre$ the function given by $E_R^w(x)=E_R^w(r(x))$. To compare the mean exit time function, we need  the following comparison for the mean exit time

\begin{proposition}(\cite{MP2}) Let $\varphi: P^m \longrightarrow N^n$ be an isometric, proper and minimal immersion of a complete non-compact Riemannian $m$-manifold $P^m$ into a complete Riemannian manifold $N^n$ with a pole $o\in N$.  Let us suppose that the $o-$radial sectional curvatures of $N$ are bounded from above by
$$K_{o,N}(\sigma_{x}) \leq  -\frac{w''(r)}{w(r)}(\varphi(x))\,\,\,\forall x \in P\, ,$$ and the model space $M_w^m$ is balanced from below, then
\begin{equation}\label{Laplacian-exit}
\Delta^PE^w_R\leq -1=\Delta^PE_R.
\end{equation}
\end{proposition}

Applying now the divergence theorem to inequality (\ref{Laplacian-exit}) we obtain

\begin{equation}
\begin{aligned}\label{anterior}
-\Vol(D_R)=&\int_{D_R}\Delta^PE^P_R(r)d\mu\geq \int_{D_R}\Delta^PE^w_R(r)d\mu\\
=&\int_{\partial D_R}E^w_R(r)'\langle \nabla^Pr, \nu\rangle d\sigma=-q_w(R) \int_{\partial D_R}\Vert \nabla^Pr\Vert d\sigma
\end{aligned}
\end{equation}

Therefore,

\begin{equation}\label{isoper-ineq}
 \frac{\Vol(D_R)}{\Vol(B^w_R)}\leq
 \frac{J_r(R)}{\Vol(S^w_R)}=\frac{J_r(R)}{J_r^w(R)}\quad .
\end{equation}

Observe that equality in inequality (\ref{isoper-ineq}) implies equality in inequality  (\ref{anterior}) and therefore, in inequality  (\ref{Laplacian-exit}). Taking, thus, into account that $E_R^P=E_R ^w$ in $x\in \partial D_R$, $\Delta E_R^P=\Delta E_R ^w $ in $x\in D_R$, and the maximum principle, we obtain that equality in (\ref{isoper-ineq}) implies
\begin{equation}
E_R^P=E_R^w\quad,
\end{equation}
for all  $x\in D_R$.

In order to obtain the monotonicity of the quotient $\frac{\Vol(D_R)}{\Vol(B^w_R)}$, we note that by the co-area formula we get:
\begin{equation}\label{monotonicity}
\begin{aligned}
\left(\ln \frac{\Vol(D_R)}{\Vol(B^w_R)}\right)'=& \frac{\int_{\partial D_R}\frac{1}{\Vert \nabla^Pr\Vert} d\sigma}{\Vol(D_R)}-\frac{\Vol(S_R^w)}{\Vol(B_R^w)}\\
\geq & \frac{\int_{\partial D_R}\Vert \nabla^Pr\Vert d\sigma}{\Vol(D_R)}-\frac{\Vol(S_R^w)}{\Vol(B_R^w)}\\
= & \frac{\Vol(S_R^w)}{\Vol(D_R)}\left(\frac{J_r(R)}{\Vol(S_R^w)}-\frac{\Vol(D_R)}{\Vol(B_R^w)}\right)\\\geq & 0\quad.
\end{aligned}
\end{equation}

Hence $\frac{\Vol(D_R)}{\Vol(B^w_R)}$ is a monotone non-decreasing function. To prove that also $\frac{J_r(R)}{J_r^w(R)}$ is a monotone nondecreasing function we need the following lemma
\begin{lemma}
\begin{equation}
\text{div }\left(\frac{\nabla ^P E^w_R(r)}{\Vol(B_r^w)}\right)\leq 0\quad.
\end{equation}
\end{lemma}
\begin{proof}
Taking into account  the product rule for the divergence and the mean exit time comparison result
\begin{equation}\label{nonpositivediv}
\begin{aligned}
\text{div }\left(\frac{\nabla ^P E^w_R(r)}{\Vol(B_r^w)}\right)=&\frac{\Delta^PE_R^w(r)}{\Vol(B_r^w)}-\frac{\Vol(B_r^w)'}{\Vol(B_r^w)^2}\langle \nabla^P r, \nabla^P E_R^w(r) \rangle\\
=&\frac{\Delta^PE_R^w(r)}{\Vol(B_r^w)}-\frac{\Vol(S_r^w)}{\Vol(B_r^w)^2}E_R^w(r)' \Vert \nabla^P r\Vert^2 \\\leq& \frac{-1}{\Vol(B_r^w)}+\frac{\Vert \nabla^P r\Vert^2}{\Vol(B_r^w)}\leq 0  \quad.
\end{aligned}
\end{equation}
\end{proof}

Using now this lemma and the divergence theorem in the extrinsic annulus $A_{\rho,R}$ for $\rho<R$
\begin{equation}\label{negdiv}
\begin{aligned}
0\geq &\int_{A_{\rho,R}} \text{div }\left(\frac{\nabla ^P E^w_R(r)}{\Vol(B_r^w)}\right)d\mu\\
 = & \int_{\partial D_R}\frac{E^w_R(r)' \Vert \nabla^P r\Vert}{\Vol(B_r^w)}d\sigma-\int_{\partial D_\rho}\frac{E^w_R(r)' \Vert \nabla^P r\Vert}{\Vol(B_r^w)}d\sigma\\
=&-\frac{J_r(R)}{\Vol(S_R^w)}+\frac{J_r(\rho)}{\Vol(S_\rho^w)}\quad.
\end{aligned}
\end{equation}
Therefore,
\begin{equation}
\frac{J_r(R)}{J^w_r(R)}\geq \frac{J_r(\rho)}{J_r^w(\rho)}\quad,
\end{equation}
for any $R>\rho$, and  the theorem is proven.

\subsection{Proof of theorem \ref{capacity-baix} }

The corresponding Dirichlet problem for the capacity of the extrinsic
annulus $A_{\rho,R}$ is
\begin{equation}
\begin{cases}
\Delta^P \Psi=0\\
\Psi\vert_{\partial D_\rho}=0\\
\Psi\vert_{\partial D_R}=1
\end{cases}
\end{equation}
 
Let us transplant the function $\Psi^w_{\rho,R}$ with the extrinsic
distance function $r$:
\begin{equation}
\Psi^w(p):A_{\rho,R}\to \erre, \quad p\to\Psi^w(p):=\Psi^w_{\rho,R}(r(p))\quad.
\end{equation}
Then, applying proposition \ref{lap-comp}
\begin{equation}
\Delta^P \Psi^w\geq m \left(1-\vert \nabla^P
  r\vert\right)\left((\Psi^w_{\rho,R})'\cdot \eta_w\right)\circ r \quad .
\end{equation}
Taking into account that $\eta_w\geq 0$
\begin{equation}
\Delta^P \Psi^w\geq 0=\Delta^P \Psi\quad.
\end{equation}
Since $\Delta^P\left(\Psi^w-\Psi\right)\geq 0$ and since
$\Psi_{\partial A_{\rho,R}}=\Psi^w_{\partial A_{\rho,R}}$, we have by the
Maximum Principle that $\Psi^w\leq \Psi$ on $A_{\rho,R}$, and, since
$\Psi_{\partial D_\rho}=\Psi^w_{\partial D_\rho}=0$, we obtain
\begin{equation}
\vert \nabla^P \Psi^w\vert \leq \vert \nabla ^P \Psi \vert \quad
\text{ on } \partial D_\rho \quad .
\end{equation}
Finally, we can estimate the capacity
\begin{equation}
\begin{aligned}
\C(A_{\rho,R})&=\int_{\partial D_\rho} \vert \nabla^P \Psi \vert d\sigma\\
&\geq \int_{\partial D_\rho}\vert \nabla^P\Psi^w \vert d\sigma\\
&=(\Psi^w(\rho))'\int _{\partial D_\rho}\vert \nabla^P r \vert d\sigma\\
&=\C(A_{\rho,R}^w)\frac{J_r(\rho)}{J^w_r(\rho)} \quad ,
\end{aligned}
\end{equation}
and the theorem follows.

\subsection{Proof of theorem \ref{capacity}}

With the flux we can provide an upper bound for the capacity (see inequality (\ref{capa-gry}) ). Using theorem \ref{isoperimetric} we obtain that
\begin{equation}
\begin{aligned}
\C(A_{\rho,R})\leq & \frac{1}{\int_\rho^R\frac{ds}{\int_{\partial D_s}\Vert \nabla^P r\Vert d\sigma}}=\frac{1}{\int_\rho^R\frac{ds}{\frac{J_r(s)}{\Vol(S_s^w)}\Vol(S_s^w)}}\\
\leq  & \frac{\frac{J_r(R)}{\Vol(S_R^w)}}{\int_\rho^R\frac{ds}{\Vol(S_s^w)}}=\frac{J_r(R)}{\Vol(S_R^w)}\C (A^w_{\rho,R}).
\end{aligned}
\end{equation}

For the bounds from below, see \cite{MPcapacity}. Observe moreover that equality in the above inequality implies that
\begin{equation}
\int_t^R \left(  \frac {\frac{J_r(R)}{\Vol(S^w_R)}}{\frac{J_r(s)}{\Vol(S^w_s)}}-1\right) \frac{1}{\Vol(S_s^w)}ds=0.
\end{equation}

Therefore
\begin{equation}
\frac{J_r(R)}{\Vol(S^w_R)}=\frac{J_r(s)}{\Vol(S^w_s)},
\end{equation}
for any $s\in [\rho,R]$. Then, by inequality (\ref{negdiv})
\begin{equation}
 \text{div }\left(\frac{\nabla ^P E^w_R(r)}{\Vol(B_r^w)}\right)=0,
\end{equation}
for any $p\in A_{\rho,R}$. From inequality (\ref{nonpositivediv})

\begin{equation}
 \Vert \nabla^P r \Vert =1,
\end{equation}
for any $p\in A_{\rho,R}$, and hence, $D_R$ is a minimal cone.

\section{Proof of Theorem \ref{flat-ends} and Corollary \ref{col-ends} }\label{Proof-into}
This proof mimics the argument given in \cite[Theorem 2]{Tkachev}, so we merely give a sketch emphasizing the points where the line of reasoning from \cite{Tkachev} is modified to hold in the present more general setting.

First of all, note that we can construct the following order-preserving bijection
$$
F:\erre^+  \to \erre^+,\quad F(t)=\int_0^t w(s)ds\quad .
$$

Since $\varphi:P^m\to N^n$ is a complete proper and minimal immersion within a manifold with a pole $N^n$, applying proposition \ref{lap-comp} we have
\begin{equation}\label{anterior-2}
\Delta^P F\circ r\geq m w'\circ r\quad .
\end{equation}
Hence, by using the assumption $w'>0$,  the extrinsic distance has no local maximum. Therefore for any $R$, $P^m\setminus D_R$ has no bounded components, being each component of $P^m\setminus D_R$  non compact, and the number of ends $\mathcal{E}_{D_R}(P)$ with respect to $D_R$ is the number of connected components of $P^m\setminus D_R$.

Let us denote  by $\displaystyle\left\{\Omega^i\right\}_{i=1}^{\mathcal{E}_{D_R}(P)}$ the set of $\mathcal{E}_{D_R}(P)$ connected components of $P^m\setminus D_{R}$ (every one of them is a  minimal submanifold with boundary).
Now we need the following lemma

\begin{lemma}For any connected component $\Omega_i$ of $P^m\setminus D_R$ the volume of the set
$$
D_t^{\Omega_i}=D_t\cap \Omega_i \quad .
$$
is bounded from below by
\begin{equation}\label{equation-lemma}
\Vol(D_t^{\Omega_i})\geq V_m \left(\frac{t-R}{2}\right)^m\quad .
\end{equation}
\end{lemma}
\begin{proof}
Now pick a point $o'\in D_t^{\Omega_i}$ such that its extrinsic distance is $r_o(o')=\frac{R+t}{2}$, then the extrinsic ball $D^{\Omega_i}_{\frac{t-R}{2}}(o')$ in $\Omega_i$ centered at $o'$ with radius  $\frac{t-R}{2}$ satisfies
\begin{equation}
D^{\Omega_i}_{\frac{t-R}{2}}(o') \subset D_t^{\Omega_i}\quad .
\end{equation}
Hence,
\begin{equation}
  \Vol(D_t^{\Omega_i})\geq \Vol(D^{\Omega_i}_{\frac{t-R}{2}}(o'))\quad.
 \end{equation}
 
 Since  $r(o')>R$ and the sectional curvatures of any tangent $2-$plane of the tangent space at every point in the geodesic ball $B^{N}_{\frac{t-R_0}{2}}(\varphi(o'))$ of the ambient manifold are non-positive, we can make use of the behavior of the volume quotient  (claim (2) in theorem \ref{isoperimetric}) for extrinsic balls to  the immersion $\varphi: D^{\Omega_i}_{\frac{t-R_0}{2}}(o') \to B^{N}_{\frac{t-R_0}{2}}(\varphi(o'))$ with the new model comparison $w(r)=w_0(r)=r$ (namely $M_w^m=\erre^m$)
 
\begin{equation}
\frac{\Vol(D^{\Omega_i}_{\frac{t-R}{2}}(o'))}{V_m \left(\frac{t-R}{2}\right)^m}\geq \lim_{s\to 0}\frac{\Vol(D^{\Omega_i}_{s}(o'))}{V_m s^m}\geq 1\quad .
\end{equation}
And the lemma is proved.
\end{proof}

Summing now in inequality (\ref{equation-lemma}) we obtain
\begin{equation}
\Vol(A_{R,t})=\sum_{i=1}^{\mathcal{E}_{D_R}(P)}\Vol(D_t^{\Omega_i})\geq \mathcal{E}_{D_R}(P) V_m \left(\frac{t-R}{2}\right)^m\quad .
\end{equation}
Taking into account that $\Vol(A_{R,t})\leq \Vol(D_t)$ and dividing by $\Vol(B_t^w)$ we obtain
\begin{equation}
\frac{\Vol(D_t)}{\Vol(B_t^w)}\geq \mathcal{E}_{D_R}(P) V_m \frac{\left(\frac{t-R}{2}\right)^m}{\Vol(B_t^w)}\quad .
\end{equation}
We can split the last quotient by division and multiplication by $t^m$
\begin{equation}
\frac{\Vol(D_t)}{\Vol(B_t^w)}\geq \mathcal{E}_{D_R}(P) V_m \left(\frac{1-R/t}{2}\right)^m\frac{t^m}{\Vol(B_t^w)}\quad .
\end{equation}
Hence, finally, using the explicit expression for $ \Vol(B_t^w)$ the theorem follows.

In order to prove corollary \ref{col-ends}, note that  by the maximum principle $\mathcal{E}_{D_R}^P$ is a non-decreasing function with respect to $R$. By inequality (\ref{resul1}) and the assumptions of the corollary we can conclude that $\mathcal{E}_{D_R}^P$ is stabilized, i.e. $\mathcal{E}_{D_R}^P= \text{ constant for sufficient large }R$.

Now let $F\subset P$ be an arbitrary compact subset. Using again the maximum principle of the immersion, we conclude that $\mathcal{E}_F(P)$ is a non-decreasing function of the compact set $F$ (namely, if $F_1\subset F_2$ then $\mathcal{E}_{F_1}(P)\leq\mathcal{E}_{F_2}(P)$). Taking into account that  for any compact set $K$ there exist $R_K$ such that $K\subset D_{R_K}$, we finally obtain
\begin{equation}
\mathcal{E}(P)=\lim_{R\to \infty} \mathcal{E}_{D_R}(P)\quad,
\end{equation}
and the corollary follows.

\section{Corollaries and application of the extrinsic comparison theory }\label{corollaries}

\subsection{Relation between $w-$volume and $w-$flux of submanifolds}

Under the hypotesis of theorem \ref{isoperimetric}, if the submanifold has finite $w$-flux, the submanifold has finite $w$-volume. But in particular settings we can also state a reverse:

\begin{theorem}\label{volume-flux-theo} Let $\varphi: P^m \longrightarrow N^n$ be an isometric, proper, and minimal immersion of a complete non-compact Riemannian $m$-manifold $P^m$ into a complete Riemannian manifold $N^n$ with a pole $o\in N$ .  Let us suppose that the $o-$radial sectional curvatures of $N$ are bounded from above by
$$K_{o,N}(\sigma_{x}) \leq  -\frac{w''(r)}{w(r)}(\varphi(x))\,\,\,\forall x \in P\, .$$
Suppose that the model space $M_w^m$ is balanced from bellow with warping function satisfying
$$
w'(r)\geq 0\quad\forall r\in \erre_+.
$$
Then, if the submanifold has finite $w$-volume, we have:
\begin{enumerate}
\item  The submanifold has finite $w$-flux.
\item $\text{Flux}_{w}(P)=\Vol_{w}(P)$.
\end{enumerate}
\end{theorem}
\begin{proof}
To prove the theorem let us state the following metric property for geodesic balls and geodesic spheres in a rotationally symmetric model space
\begin{lemma}\label{bola-s}
Let $M^m_w$ be a model space with
$$
w'(r)\geq 0\quad\forall r\in \erre_+.
$$
Then
$$
q_w(s)=\frac{\Vol(B^w_S)}{\Vol(S^w_s)}\leq s.
$$
\end{lemma}
\begin{proof}
Observe that
\begin{equation}
q_w(0)=0,
\end{equation}
and, since $w'\geq 0$ ,
\begin{equation}
q_w'(t)\leq 1,\, \forall t\geq 0.
\end{equation}
Hence, by integrating the above inequality, the lemma follows.
\end{proof}

Now, since $P$ has finite $w$-volume, then there exists $S\in \erre^+$ such that
\begin{equation}
\frac{\Vol(D_R)}{\Vol(B_R^{w})}\leq \lim_{t\to \infty}\frac{\Vol(D_t)}{\Vol(B_t^{w})}=S<\infty\quad.
\end{equation}
By inequality (\ref{monotonicity})
\begin{equation}
\begin{aligned}
\left(\ln \frac{\Vol(D_s)}{\Vol(B^{w}_s)}\right)'\geq&\frac{\Vol(S_s^{w})}{S\Vol(B_s^{w})}\left(\frac{J_r(s)}{J_r^{w}(s)}-\frac{\Vol(D_s)}{\Vol(B_s^{w})}\right)\geq 0 \quad.
\end{aligned}
\end{equation}
Therefore, taking lemma \ref{bola-s} into acount we get:
\begin{equation}
\begin{aligned}\label{volume-flux}
0\leq \left(\frac{J_r(s)}{J_r^{w}(s)}-\frac{\Vol(D_s)}{\Vol(B_s^{w})}\right)\leq S\left(\ln \frac{\Vol(D_s)}{\Vol(B^{w}_s)}\right)' s\quad.
\end{aligned}
\end{equation}
But since
\begin{equation}
\lim_{t\to \infty}\frac{\Vol(D_t)}{\Vol(B_t^{w})}=\frac{\Vol(D_{R_1})}{\Vol(B_{R_1}^{w_b})}e^{\int_{R_1}^\infty\left(\ln \frac{\Vol(D_s)}{\Vol(B^{w}_s)}\right)' ds}=S<\infty\quad,
\end{equation}
then, for any $\epsilon>0$ there exist a sequence $\{t_i\}_{i=1}^\infty$ with $t_i\to \infty$ when $i\to \infty$, and $R_\epsilon$, such that
\begin{equation}
\left(\ln \frac{\Vol(D_{t_i})}{\Vol(B^{w}_{t_i})}\right)' t_i<\epsilon\quad.
\end{equation}
This holds for any $t_i>R_\epsilon$. Applying inequality (\ref{volume-flux}) taking into account the monotonicity of the flux and volume comparison quotients
\begin{equation}
0\leq \text{Flux}_{w_b}(P)-\Vol_{w_b}(P)\leq \epsilon\quad ,
\end{equation}
for any $\epsilon>0$ . Letting $\epsilon$ tend to $0$, the theorem is proven.
\end{proof}

\subsection{Intrinsic versions}

In this subsection we consider the intrinsic versions of Theorems
\ref{isoperimetric} and \ref{capacity} assuming that $P^m=N^n$. In this case, the extrinsic
distance to the pole $p$ becomes the intrinsic distance in $N^n$,
hence, for all $R$ the extrinsic domains $D_R$ become the geodesic
balls $B^N_R$ of the ambient manifold $N^n$. Then, for all $x \in P$
\begin{eqnarray*}
\nabla^P r(x)&=&\nabla^N r (x).
\end{eqnarray*}
As a consequence, $\|\nabla^P r\|=1$.

From this intrinsic viewpoint, we have the following
isoperimetric and volume comparison inequalities.

\begin{theorem} Let $N^n$ denote a complete
Riemannian manifold with a pole $p$. Suppose that the $p$-radial
sectional curvatures of $N^n$ are bounded from above by the
$p_w$-radial sectional curvatures of a $w$-model space $M^n_w$.
Assume that
\begin{equation}
w'\geq 0\quad.
\end{equation}
Then the capacity of the intrinsic annulus $A_{\rho,R}$ is bounded from below by
$$
\frac{\Vol(\partial B^N_\rho)}{\Vol(S_\rho^w)} \leq \frac{\C(A_{\rho,R})}{\C(A_{\rho,R}^w)}
$$
And, furthermore, if $M_w^n$ is hyperbolic, then $N^n$ is also hyperbolic.
\end{theorem}

\begin{theorem} Let $N^n$ denote a complete
Riemannian manifold with a pole $p$. Suppose that the $p$-radial
sectional curvatures of $N^n$ are bounded from above by the
$p_w$-radial sectional curvatures of a $w$-model space $M^n_w$.
Assume that $M_w^n$ is balanced from below.
Then,
\begin{enumerate}
\item for all $R>0$
\begin{equation}\label{equ3}
\frac{\Vol( B^N_R)}{\Vol( B^w_R)} \leq \frac{\Vol(\partial B^N_R)}{\Vol(S^w_R)}\quad.
\end{equation}
\item The functions $\frac{\Vol( B^N_R)}{\Vol( B^w_R)}$ and $ \frac{\Vol(\partial B^N_R)}{\Vol(S^w_R)}$ are non decreasing on $R$.
\item Denoting by $E_R^N(x)$ the mean exit time function for the geodesic ball $B^N_R$ in $N$ and denoting by $E_R^w$ the mean exit time function in the $R-$ball $B_R^w$ in the model space $M_w^n$. If equality holds in (\ref{equ3}) for some fixed $R>0$ then for any $x\in B^N_R$,  $E_R^N(x)=E_R^w(r(x))$.
\item The capacity of the intrinsic annulus $A_{\rho,R}$ is bounded from above by
\begin{displaymath}
 \frac{\C(A_{\rho,R})}{\C(A_{\rho,R}^w)}\leq \frac{\Vol(\partial B^N_R)}{\Vol(S_R^w)}
\end{displaymath}
\end{enumerate}
Furthermore, if we suppose that there exist a finite real constant $C<\infty$ such that $\frac{\Vol(B^N_R)}{\Vol(B_R^w)} <C$ (or $\frac{\Vol(\partial B^N_R)}{\Vol(S_R^w)} <C$) then if $M_w^n$ is parabolic, $N$ is parabolic, and
 \begin{displaymath}
\lim_{R\to\infty}\frac{\Vol(B^N_R)}{\Vol(B_R^w)}=\lim_{R\to\infty}\frac{\Vol(\partial B^N_R)}{\Vol(S_R^w)}.
\end{displaymath}

\end{theorem}

\subsection{Upper bounds for the fundamental tone}
S.T Yau suggested in \cite{YauReview} the ``very interesting'' question to find an upper estimate to the first Dirichlet eigenvalue of  minimal surfaces.
 
Recall that for any precompact region $\Omega\subset M$ in a Riemannian manifold $M$, the first eigenvalue $\lambda_1(\Omega)$ of the Dirichlet problem in $\Omega$ for the Laplace operator is defined by the variational property
\begin{equation}
\lambda_1(\Omega)=\inf_u \frac{\int \Vert \nabla u \Vert^2d\mu}{\int \Vert u \Vert^2}
\end{equation}
where the $\inf$ is taken over all Lipschitz functions $u\neq 0$ compactly supported in $\Omega$.

The fundamental tone $\lambda^*(M)$ of a complete Riemannian manifold can be obtained as the limit of the first Dirichlet eigenvalues of the precompact open sets in any exhaution sequence $\{\Omega_n\}_{n\in \mathbb{N}}$ for $M$, see \cite{GriExp}
\begin{equation}
\lambda^*(M)=\lim_{n\to\infty}\lambda_1(\Omega_n)\quad.
\end{equation}

In this section, we shall impose flux and volume restrictions not on the submanifold $P$ but on one end $V$ of the submanifold with respect to the extrinsic ball $D_{R_0}$.
Let us denote $D_R^{V}$ the intersection of the extrinsic ball $D_R$ with the end $V$ with respect to $D_{R_0}$
\begin{equation}
D_R^{V}=D_R \cap V \quad.
\end{equation}

Let us denote $J^V_r(R)$ the flux of the extrinsic distance in the end $V$, namely
\begin{equation}
J^V_r(R)=\int_{\partial D_R\cap V}\vert \nabla^P r \vert d\sigma\quad .
\end{equation}

With this setting we then have:
\begin{theorem}\label{fundamental-tone}
Let $\varphi: P^m \longrightarrow N^n$ be an isometric, proper and minimal immersion of a complete non-compact Riemannian $m$-manifold $P^m$ into a complete Riemannian manifold $N^n$ with a pole $o\in N$.  Let us suppose that the $o-$radial sectional curvatures of $N$ are bounded from above by
$$K_{o,N}(\sigma_{x}) \leq  -\frac{w''(r)}{w(r)}(\varphi(x))\,\,\,\forall x \in P\, ,$$ and the model space $M_w^m$ is balanced from below.  Suppose moreover that there exists an end $V$ with respect to an extrinsic ball $D_{R_0}$ with finite $w$-flux. Then
\begin{equation}
\lambda^*(P)\leq \frac{\text{Flux}_w(V)}{\Vol_w(V)}\limsup_{t\to\infty} \left(\frac{1}{\Vol(B_{t}^w)\int_{t}^\infty \frac{ds}{\Vol(S^w_s)}}\right)\quad .
\end{equation}
\end{theorem}

\begin{proof}
Due to the relation between the first Dirichlet eigenvalue and the capacity given in \cite{GriCap} we can conclude for the extrinsic ball $D_R^V$ that
\begin{equation}
\lambda_1(D_R^V)\leq \frac{\C(A^V_{t,R})}{\Vol(D_t^V)}\quad .
\end{equation}
Being $t<R$ and $A^V_{t,R}$ the extrinsic annulus in $V$. Hence, by the theorem \ref{capacity}
\begin{equation}
\lambda_1(D_R^V)\leq \frac{\frac{J_r^V(R)}{J_r^w(R)}}{\frac{\Vol(D_t^V)}{\Vol(B_t^w)}} \frac{\C(A^w_{t,R})}{\Vol(B^w_t)}\quad .
\end{equation}
For any $t<R$. Finally, taking into account that $\lambda(D_R)\leq \lambda_1(D_R^V)$ (by the monotonicity of the first eigenvalue), and letting  $R$ tend to infinity we have

\begin{equation}
\lambda^*(P)\leq \frac{\text{Flux}_w(V)}{\frac{\Vol(D_t^V)}{\Vol(B_t^w)}} \frac{1}{\Vol(B^w_t)\int_t^\infty\frac{ds}{\Vol(S_s^w)}}\quad .
\end{equation}
Taking limits, the theorem follows.
\end{proof}

Obviously, by using theorem \ref{volume-flux-theo} we also have the following:

\begin{corollary}Under the assumptions of theorem \ref{fundamental-tone} suppose moreover
$$
w'\geq 0.
$$
Then,
\begin{equation}
\lambda^*(P)\leq \limsup_{t\to\infty} \left(\frac{1}{\Vol(B_{t}^w)\int_{t}^\infty \frac{ds}{\Vol(S^w_s)}}\right)\quad .
\end{equation}
\end{corollary}

Using the Cheeger isoperimetric constant we can deduce the following lower bounds
\begin{theorem}\label{fundamental-tone-b}
Let $\varphi: P^m \longrightarrow N^n$ be an isometric, proper and minimal immersion of a complete non-compact Riemannian $m$-manifold $P^m$ into a complete Riemannian manifold $N^n$ with a pole $o\in N^n$ .  Let us suppose that the $o-$radial sectional curvatures of $N^n$ are bounded from above by
$$K_{o,N}(\sigma_{x}) \leq  -\frac{w''(r)}{w(r)}(\varphi(x))\,\,\,\forall x \in P\, ,$$ and the model space $M_w^m$ is balanced from below.  Suppose moreover that
$$
L:=\sup_{t\in \erre^+} q_w(t)<\infty.
$$
 Then
\begin{equation}
\frac{1}{4L^2}\leq \lambda^*(P)\quad .
\end{equation}
\end{theorem}
\begin{proof}
Consider $\Omega\subset P^m$ a smooth domain with smooth boundary $\partial \Omega$. Using the transplanted mean exit function in a similar way as in the proof of theorem \ref{isoperimetric} we obtain:

\begin{equation}\begin{aligned}
-\Vol(\Omega)=&\int_\Omega\Delta^PE_Rd\mu\geq \int_\Omega\Delta^PE_R^wd\mu=\int_{\partial \Omega}E_R^w(r)'\langle \nabla ^Pr, \nu\rangle d\sigma\\
\geq &-\int_{\partial \Omega}q_w(r)\langle \nabla^Pr,\nu\rangle d\sigma\geq-\int_{\partial \Omega}q_w(r) d\sigma\\
\geq & - L\Vol(\partial \Omega)\quad .
\end{aligned}
\end{equation}
Hence, for any $\Omega\subset P$,
\begin{equation}
\frac{\Vol(\partial \Omega)}{\Vol (\Omega)}\geq \frac{1}{L}\quad.
\end{equation}
Thence the Cheeger constant $h(P)$ (see \cite{Chavel}) satisfies
\begin{equation}
h(P)\geq \frac{1}{L}\quad .
\end{equation}
Taking into account that
\begin{equation}
\lambda^*(P)\geq \frac{1}{4} \left(h(P)\right)^2\quad ,
\end{equation}
the theorem follows.
\end{proof}

As an immediate consequence of the previous theorems and corollaries in the particular setting of a minimal submanifold within a Cartan-Hadamard ambient manifold is the following:

\begin{corollary}Let $\varphi: P^m \longrightarrow N^n$ be a complete minimal immersion  into a simply connected Cartan-Hadamard manifold $N^n$ with sectional curvatures $K_N\leq b \leq 0$.  Suppose moreover that there exists an end $V$ with respect to an extrinsic ball $D_{R_0}$ with finite $w_b$-volume. Then
\begin{equation}
\frac{-(m-1)^2b}{4}\leq \lambda^*(P)\leq -(m-1)^2b\quad.
\end{equation}
\end{corollary}

\begin{remark}
Note that if $b=0$ in the above theorem, $\lambda^*(P)=0$. See also \cite{GimFT}.
\end{remark}

\subsection{Applications to minimal submanifolds in $\erre^n$}
If $P^m$ is a minimal submanifold in $\erre^n$, it is well known that the extrinsic distance $r$ satisfies

\begin{equation}
\Delta^Pr^2=2m
\end{equation}
Applying the divergence theorem
\begin{equation}
\begin{aligned}
2m \Vol(D_R)=&\int_{D_R}\Delta^pr^2 d\mu=\int_{\partial D_R}2r \langle \nabla r, \nu\rangle d\sigma\\
=&2R\int_{\partial D_R}\langle \nabla r, \frac{\nabla r}{\vert \nabla r\vert} \rangle d\sigma=2R\int_{\partial D_R}\vert \nabla r\vert d\sigma\\
=&2m\frac{\Vol(B_R^{w_0})}{\Vol(S_R^{w_0})}\int_{\partial D_R}\vert \nabla r\vert d\sigma
\end{aligned}
\end{equation}

hence, the volume comparison quotient $\frac{\Vol(D_R)}{\Vol(B_R^{w_0})}$ is just
\begin{equation}\label{equal-flux-vol}
\frac{\Vol(D_R)}{\Vol(B_R^{w_0})} =\frac{J_r(R)}{J^{w_0}_r(R)}\quad .
\end{equation}

And therefore, we can state that
\begin{corollary}Let $P^m$ be a minimal submanifold properly immersed in the Euclidean space $\erre^n$. Then
$$
E^P_R(x)=E^{\erre^m}_R(r(x))\quad,
$$
where $E^P_R(x)$ denotes the first exit from $D_R$ for a Brownian particle starting at $x\in D_R$, and  $E_R^w(r)$ denotes the (rotationally symmetric) mean exit time function  for the $R-$ball $B_R^w$ in the model space $M_w^m$
\end{corollary}

If we have finite $w_0$-volume  ($\sup_{R\in \erre^+}\frac{\Vol(D_R)}{\Vol(B_R^{w_0})}<\infty$)
we also get:

\begin{corollary}\label{lambda-1}
Let $P^m$ be a minimal submanifold immersed in $\erre^n$, suppose moreover that $P$ has finite $w_0$-volume  then:
\begin{enumerate}
\item $P$ is parabolic if  $m=2$ and if $m\geq 3$, $P$ is hyperbolic.
\item $\lambda^*(P)=0$.
\end{enumerate}
\end{corollary}

On the other hand, in special geometric settings  the finiteness of the $w_0$-volume is related to the number of ends

\begin{theoremA}(See \cite{A1} and \cite{ChSubv})
Let $P^m$ be a minimal submanifold properly immersed in $\erre^n$ with finite total scalar curvature i.e. $\int_P\Vert B^P \Vert^m d\mu <\infty$ where  $\Vert B^P \Vert$ denotes the norm of the second fundamental form in $P$, then
\begin{equation}
\frac{J_r(R)}{J_r^{w_0}(R)}\leq \mathcal{E}(P),
\end{equation}
provided either of the following two conditions hold
\begin{enumerate}
\item $m=2$, $n=3$ and each end of $P$ is embedded.
\item $m\geq 3$.
\end{enumerate}
Where  $\mathcal{E}(P)$ denotes the finite number of ends of $P$.
\end{theoremA}

This relation between the number of ends and the flux quotient allow us to state

\begin{corollary}\label{capacity-ends}Let $P^m$ be a minimal submanifold properly immersed in $\erre^n$ with finite total scalar curvature and either $m\geq 3$, or $m=2$ $n=3$ and each end of $P$ is embedded, then
for any $\rho>0$ and any $R> \rho$
\begin{equation}
1\leq \frac{\C(A_{\rho,R})}{\C(A_{\rho,R}^w)}\leq \mathcal{E}(P)\quad .
\end{equation}
And for the fundamental tone
\begin{equation}
\lambda^*(P)=0\quad.
\end{equation}
\end{corollary}


\def\cprime{$'$} \def\cprime{$'$} \def\cprime{$'$} \def\cprime{$'$}
  \def\cprime{$'$}
\providecommand{\bysame}{\leavevmode\hbox to3em{\hrulefill}\thinspace}
\providecommand{\MR}{\relax\ifhmode\unskip\space\fi MR }
\providecommand{\MRhref}[2]{%
  \href{http://www.ams.org/mathscinet-getitem?mr=#1}{#2}
}
\providecommand{\href}[2]{#2}

\end{document}